\documentclass[11pt, reqno]{article}
\usepackage{amsmath, amsthm, amscd, amsfonts, amssymb, graphicx, color}
\usepackage{graphicx}
 \usepackage[vmargin=2cm,left=2cm,right=2cm]{geometry}
\begin{document}
\setcounter{page}{1}
\setlength{\unitlength}{12mm}
\newcommand{\f}{\frac}
\newtheorem{theorem}{Theorem}[section]
\newtheorem{lemma}[theorem]{Lemma}
\newtheorem{proposition}[theorem]{Proposition}
\newtheorem{corollary}[theorem]{Corollary}
\theoremstyle{definition}
\newtheorem{definition}[theorem]{Definition}
\newtheorem{example}[theorem]{Example}
\newtheorem{solution}[theorem]{Solution}
\newtheorem{xca}[theorem]{Exercise}
\theoremstyle{remark}
\newtheorem{remark}[theorem]{Remark}
\numberwithin{equation}{section}
\newcommand{\sta}{\stackrel}
\title{\bf Diaz-Metcalf Type Inequality for Sugeno and Pseudo Integrals}
\author {Mohammad Reza Karimzadeh$^{a}$, Bayaz Daraby$^{b}$, Asghar Rahimi$^{c}$}
\date{\footnotesize $^{a, b, c}$Department of Mathematics, University of Maragheh, P. O. Box 55136-553, Maragheh, Iran
}
\maketitle

\begin{abstract}
\noindent
In this paper, we have proved Diaz-Metcolf inequality for fuzzy integrals. More precisely:
\\
If $f, g: [0, 1]\to\mathbb{R}$ are continuous and strictly increasing functions, then the inequality
$$ - \hspace{-1em} \int_0^1 f^s d\mu\cdot  - \hspace{-1em} \int_0^1  g^sd\mu\le  - \hspace{-1em} \int_0^1\left(f\cdot g\right)^sd\mu,$$
holds, where $s>1$ and $\mu$ is  the Lebesque measure on $\mathbb{R}$. In addition, we have shown this inequality for pseudo-integrals.

\vspace{.5cm}\leftline{{Subject Classification 2010: 03E72, 26E50, 28E10.
}}

\vspace{.5cm}\leftline{Keywords: Diaz-Metcalf type inequality, Fuzzy integral, Integral inequality, Pseudo integral.}
\end{abstract}
\section{Introduction}

Fuzzy measure and Sugeno integral, which were originally introduced by Sugeno in 1974
 \cite{20}, are important analytical methods of measuring uncertain information \cite{p4, p2}.
\\
The study of inequalities for the Sugeno integral, which was initiated by Rom\'{a}n-Flores et al., was discussed in several papers.
Recently, the Sugeno integral counterparts of several classical inequalities, including Markov's, Chebyshev's, Jensen's, Minkowski's, H\"{o}lder's and Hardy's inequalities, were given by Flores-Franuli\v{c} and Rom\'{a}n-Flores \cite{5, 6}.
\begin{theorem}\label{t1}
\cite{Ouyang3}
Let $\mu$ be a Lebesque measure on $\mathbb{R}$ and $g: [0, \infty)\to[0, \infty)$ be a continuous and strictly increasing function and
$- \hspace{-1em} \int_0^a gd\mu=p$, then for all $a\ge 0$ we have
$$g(a-p)\ge - \hspace{-1em} \int_0^a gd\mu=p.$$
Moreover, if $0<p<a$, then
$g(a-p)=p$.
\end{theorem}

\begin{theorem}\label{t2}\cite{Ouyang3}
Let $\mu$ be a Lebesque measure on $\mathbb{R}$ and $g: [0, \infty)\to[0, \infty)$ be a continuous and strictly decreasing function and $- \hspace{-1em} \int_0^a gd\mu=p$, then for all $a\ge 0$, we have
$$g(p)\ge - \hspace{-1em} \int_0^a gd\mu=p.$$
Moreover, if $0<p<a$, then
$g(p)=p$.
\end{theorem}

We state Diyaz-Metcalf inequality in classical case is the following form.

\begin{theorem}\cite{di}
Let $(X, \mu)$ be measurement space and $f, g\in L^2(X, \mu)$ and $0\le mg\le f\le Mg$. Then
\begin{eqnarray}
\label{diaz}
\int_X f^2d\mu\cdot\int_X g^2d\mu\le\dfrac{(M+m)}{4Mm}\left(\int_X \left(f\cdot g\right)d\mu\right)^2.
\end{eqnarray}

\end{theorem}

With an example, we have shown that the classical version of the Diaz-Metcalf inequality is not valid for Sugeno integrals, and we have shown the Fuzzy form of the inequality, as well.

Pseudo-analysis is a generalization of the classical analysis, where instead of the field of real numbers a semiring is taken on a real interval $[a, b]\subseteq[-\infty, +\infty]$ endowed with pseudo-addition $\oplus$ and with pseudo-multiplication $\odot$.
Many authors such as H. Agahi, R. Mesiar and Y. Ouyang in \cite{aga} studied Chebyshev's inequality for pseudo-integrals as follows:

\begin{theorem}
Let $u, v:[0, 1]\to [a, b]$ are measurable function and $g: [a, b]\to[0, \infty)$ generator function for pseud- operation for $\oplus$ and $odot$. If $u$ and $v$ comonotone functions, we have
$$\int_{[0, 1]}^\oplus (u\odot v)dx\ge\left(\int_{[0, 1]}^\oplus udx\right)\odot \left(\int_{[0, 1]}^\oplus vdx\right).$$
\end{theorem}

Recently, B. Daraby et. al., in \cite{a, b, dar2,  bre, c} studied and generalized some other inequalities for the Fuzzy and Pseudo integrals, for example, he stated and proved \cite{2-DDD} the Stolarsky type of inequality for Pseudo integrals as it follows:

\begin{theorem}
Let $a, b>0$, $f:[0, 1]\to[0, 1]$ be a continuous and strictly decreasing function and $\mu$ be the Lebesque measure on $\mathbb{R}$ Then
$$- \hspace{-1em} \int_0^1 f\left(x^{\frac{1}{(a+b)}}\right)d\mu\ge\left(- \hspace{-1em} \int_0^1 f\left(x^{\frac{1}{a}}\right)d\mu\right)\left(- \hspace{-1em} \int_0^1 f\left(x^{\frac{1}{b}}\right)d\mu\right).$$
\end{theorem}

We have also organized in this article the following: In Section 2, we have described the definitions, properties and results of fuzzy measure, fuzzy integrals and pseudo-integrals. In Section 3,  the first  we have shown with an example that the classical Diaz-Metcalfe inequality is not valid in the fuzzy case, and then we state and prove the fuzzy and pseudo integrals cases. Finally, in Section 4, we have completly disxussed the Diaz-Metcalf inequality.

 \section{Preliminary}
 We denote by $\mathbb{R}$, the set of all real numbers. Let $X$ be a non-empty set and $\Sigma$ be a $\sigma-$algebra of subsets of $X$. Throughout this paper, all considered subsets are supposed to be in $\Sigma$.

 \begin{definition}(Ralescu and Adams \cite{12}).
A set function $\mu:\Sigma\to[0, +\infty]$ is called a fuzzy measure if the following properties are satisfied:
\begin{enumerate}
\item[(FM1)] $\mu(\emptyset)=0 $,
\item[(FM2)]
 $A \subseteq B\Rightarrow\mu(A)\leq \mu(B)$,
\item[(FM3)]
 $A_1 \subseteq A_2 \subseteq\ldots\Rightarrow\lim \mu(A_i)=\mu\left(\bigcup \limits_{i=1}^\infty A_i\right)$,
\item[(FM4)]
  $A_1 \supseteq A_2 \supseteq \ldots$ and $\mu(A_1)<\infty\Rightarrow\lim \mu(A_i)=\mu\left(\bigcap\limits_{i=1}^\infty A_i\right)$.
\end{enumerate}
When $\mu$ is a fuzzy measure, the triple $(X, \Sigma,\mu)$ is called a fuzzy measure space.
\end{definition}

If $f$ is a non-negative real-valued function on $X$, we will denote
$$F_{\alpha}=L_{\alpha}f =\left\{x\in X \mid f(x)\geq\alpha\right\}=\left\{f\geq\alpha\right\},$$ the
$\alpha$-level of $f$, for $\alpha>0.$
$L_{0}f =\overline{\left\{x\in X \mid
f(x)>0\right\}}={\rm supp}(f)$  is the  support of $f$. We know that:
 $$\alpha \leq \beta \Rightarrow  \left\{f \geq
\beta\right\} \subseteq \left\{f\geq\alpha \right\}.$$
If $\mu$ is a  fuzzy measure on $(X, \Sigma)$, we define the
following:
\begin{eqnarray*}
&& \mathfrak {F}^{\sigma}(X)=\left\{f : X\rightarrow[0,\infty)|f~ {\rm is}\ \mu-{\rm measurable}\right\}.
\end{eqnarray*}

\begin{definition}
(Wang and Klir \cite{21}).
 Let $\mu$ be a fuzzy measure on
$(X, \Sigma)$. If $f \in \mathfrak{F}^\mu(X)$ and $A\in \Sigma$, then
the Sugeno integral of $f$ on $A$, with respect
to the fuzzy measure $\mu$, is defined as
$$ - \hspace{-1em} \int_A fd\mu =\bigvee_{\alpha \geq 0} \left(\alpha \wedge\mu(A\cap F_\alpha)\right).\hspace{-1em}$$
Where $\vee$ and $ \wedge$ denotes the operations $sup$ and $inf$ on
 $[0,\infty]$, respectively. In particular, if $A=X$ then
$$ - \hspace{-1em} \int_{X} fd\mu =- \hspace{-1em} \int fd\mu=\bigvee_{\alpha \geq 0} \left(\alpha\wedge\mu(F_\alpha)\right).$$
\end{definition}

The following properties of the Sugeno integral can be found in \cite{ 21}.

\begin{proposition}\label{p23}
 (Wang and Klir \cite{21}).
  Let $(X, \Sigma, \mu)$ be a fuzzy measure space, with $A, B \in \sum$ and $f, g \in  \mathfrak {F}^{\sigma}(X)$. We have
\begin{enumerate}
\item[(1)]
 $- \hspace{-.9em} \int_A f d \mu \leq \mu(A)$.
\item[(2)]
 $- \hspace{-.9em} \int_A k d \mu =k \wedge \mu(A),$ for $k$ non-negative constant.
\item[(3)]
 If $A \subset B,$ then $- \hspace{-.9em} \int_A fd \mu \leq - \hspace{-.9em} \int_B fd \mu$.
\item[(4)]
$- \hspace{-.9em} \int_{A\cup B} fd\mu\ge - \hspace{-.9em} \int_A fd\mu\vee - \hspace{-.9em} \int_B fd\mu$.
\item[(5)]
 If $\mu(A)< \infty,$ then $- \hspace{-.9em} \int_A fd \mu \geq \alpha \Leftrightarrow \mu(A \cap \{f \geq \alpha\}) \geq \alpha$.
\item[(6)]
$- \hspace{-.8em} \int_A fd \mu \leq \alpha \Leftrightarrow \mu(A\cap F_{\alpha^+})\le \alpha\Leftarrow
 \mu(A \cap \{f \geq \alpha\}) \leq \alpha $.
 \item[(7)]
 $- \hspace{-.8em} \int_A fd \mu > \alpha \Leftrightarrow \mu(A\cap F_{\alpha^+})> \alpha\Rightarrow  \mu(A \cap \{f \geq \alpha\}) \leq \alpha $.
\end{enumerate}
\end{proposition}

 In the following, we are going to review some well known results of pseudo-operations, pseudo-analysis and pseudo-additive measures and integrals in details, we refer to \cite{mes, 21}.

Let $[a, b]$ be a closed (in some cases can be considered
semi-closed) subinterval of $[-\infty, \infty]$. The full order on
$[a, b]$ will be denoted by $\preceq$.

\begin{definition}
(Wang and Klir \cite{21}). The operation $\oplus$ (pseudo-addition)
is a function $\oplus : [a, b] \times [a, b] \rightarrow [a, b]$
which is commutative, non-decreasing (with respect to $\preceq$ ),
associative and with a zero (neutral) element denoted by {\bf 0},
i.e., for each $x \in [a, b], {\bf 0} \oplus x = x$ holds (usually
{\bf 0} is either $a$ or $b$).
\end{definition}
Let $[a, b]_+ = \left\{x | x \in [a, b], {\bf 0} \preceq x\right\}$.
\begin{definition}(Wang and Klir \cite{21}).
The operation $\odot$
(pseudo-multiplication) is a function $\odot : [a, b] \times [a, b]\to [a,b]$ which is commutative, positively non-decreasing,
i.e.,
 $x \preceq y$ implies $x \odot z \preceq y \odot z$ for all
 $z\in [a, b]_+$, associative and for which there exists a unit element
${\bf 1} \in [a, b]$, i.e., for each $x \in [a, b], {\bf 1} \odot x
= x$.
\end{definition}
We assume also ${\bf 0} \odot x = {\bf 0}$  and that $\odot$ is a
distributive pseudo-multiplication with respect to $\oplus$, i.e.,
$x \odot (y \oplus z) = (x \odot y ) \oplus (x \odot z )$.

We shall consider the semiring $([a, b], \oplus, \odot )$ for two important (with completely different behavior) cases. The first case is when pseudo-operations are generated by a monotone and continuous function $g:[a, b]\to[0, \infty)$, i.e., pseudo-operations are given with:
\begin{eqnarray}
\label{g}
x \oplus y = g^{-1}(g(x) + g(x))\qquad \text{and} \qquad x
\odot y = g^{-1}(g(x)g(x)).
\end{eqnarray}
Then, the pseudo-integral for a function $f:[c, d]\to[a, b]$ reduces on the $g-$integral
\begin{eqnarray}\label{31}
\int_{[c, d]}^\oplus f(x)dx = g^{-1}\left(\int_c^d g(f(x))dx\right).
\end{eqnarray}
More details on this structure as well as corresponding measures and integrals can be found in \cite{mes}. The second class is when $x \oplus y= \max(x, y)$ and $x \odot y= g^{-1} (g(x)g(y)),$
the  pseudo-integral for a function
$f : \mathbb{R}  \to [a,b]$ is given by
$$\int_\mathbb{R}^\oplus f \odot dm= \sup_{x\in\mathbb{R}} \left( f(x) \odot \psi(x) \right),$$
  where function $\psi$ defines sup-measure $m.$ Any sup-measure generated as essential supremum of a continuous density can be obtained as a limit of pseudo-additive measures with respect to generated pseudo-additine. Then any continuous function $f : [0, \infty] \to [0, \infty]  $ the integral $\int ^\oplus f\odot dm$ can be obtained as a limit of g-integrals.

We denote by $\mu$ the usual Lebesgue measure on $\mathbb{R}$. We have $$ m(A)= {\rm ess~sup}_{\mu}(x|x \in A)=\sup \left\{a| \mu ({x|x \in A, x >a}) >0\right\}.$$

 \begin{theorem}\label{t1}
 (Mesiar  and Pap \cite{mes}). Let $m$ be a $\sup$-measure on $([0, \infty], \mathbb{B}[0, \infty])$, where $\mathbb{B}([0, \infty])$ is the Borel $\sigma$-algebra on $[0, \infty]$,   $m(A)= {\rm ess~sup}_\mu(\psi(x)|x \in A)$, and $\psi:[0,   \infty] \to [0, \infty]$ is a  continuous.   Then for any pseudo-addition $\oplus$ with a generator $g$ there exists a family $m_\lambda$ of $\oplus_\lambda$-measure on $([0, \infty], \mathbb{B})$, where $\oplus_\lambda$ is a generated by $g^{\lambda}$ (the function $g$ of the power $\lambda, \lambda \in (0, \infty)$) such that $\lim\limits_{\lambda \to \infty} m_\lambda =m.$
\end{theorem}

\begin{theorem}\label{t2}
(Mesiar and Pap \cite{mes}). Let $([0, \infty], \sup, \odot)$ be a semiring  , when $\odot$ is a generated with $g$, i.e., we have $x \odot y =g^{-1}(g(x)g(y))$ for every $x, y \in (0, \infty).$ Let $m$ be the same as in Theorem \ref{t120}, Then there exists a family  $\{m_\lambda\}$ of  $\oplus_\lambda$     -measures, where $\oplus_\lambda$  is a generated by $g^\lambda, \lambda \in(0, \infty)$ such that for every continuous function $f:[0, \infty] \to [0, \infty],$
\begin{eqnarray}
\label{00}
\int ^{\sup} f \odot dm = \lim _{\lambda \to \infty} \int ^{\oplus _\lambda} f \odot dm_ \lambda = \lim _{\lambda \to \infty}(g^{\lambda})^{-1} \left(\int g^ \lambda (f(x))dx \right).
\end{eqnarray}
\end{theorem}

\begin{theorem}\label{t4.1}
(Mesiar and Pap \cite{mes}).
Let $u, v: [0, 1]\to [a, b]$ are two measurable functions and let a generator $g:[a, b]\to[0, \infty)$ of the pseudo addition $\oplus$ and the pseudo-multiplication $\odot$ be an increasing function and $\varphi:[a, b]\to[a, b]$ be a  continuous and strictly increasing function such that $\varphi$ commutes with $\odot$. If $u$ and $v$ are comonotone, then the inequality
$$\varphi^{-1}\left(\int_{[0, 1]}^{\oplus} \varphi(u\odot v)dx\right)\ge \left(\varphi^{-1}\left(\int_{[0, 1]}^{\oplus}\varphi(u)dx\right)\right)\odot \left(\varphi^{-1}\left(\int_{[0, 1]}^\oplus \varphi(v)dx\right)\right),$$
holds and the reverse inequality holds whenever $u$ and $v$ are countermonotone functions.
\end{theorem}

\section{Diaz-Metcalf inequality for fuzzy integrals}
At the first, by an example we show that the  inequality \eqref{diaz} is not valid for Sugeno integrals.

\begin{example}
Let
$f(x)=\dfrac{x}{2}$, $g(x)=x$ and
$0\le \dfrac{f}{g}\le 1$. By a simple calculation we get
\begin{eqnarray*}
&& - \hspace{-1.1em}\int_0^1 fd\mu=0.33,\qquad
- \hspace{-1.1em}\int_0^1 f^2d\mu=0.18,\qquad
- \hspace{-1.1em}\int_0^1 gd\mu=0.5,\\&&
- \hspace{-1.1em}\int_0^1 g^2d\mu=0.618,\qquad\quad \left(- \hspace{-1.1em}\int_0^1(f\cdot g)d\mu\right)^2=0.072 .
\end{eqnarray*}
Now, by choicing
$M=1$
and
$m=\dfrac{1}{2}$
we have
$$0.111=- \hspace{-1.1em}\int_0^1 f^2d\mu\cdot- \hspace{-1.1em}\int_0^1 g^2d\mu \nleq\dfrac{(M+m)^2}{4Mm}\left(- \hspace{-1.1em}\int_0^1\left(f\cdot g\right)d\mu\right)^2=0.081 .$$
So, the inequality \eqref{diaz} is not valid.
\end{example}

\begin{theorem}\label{t2.33}
Let $f, g: [0, 1]\to\mathbb{R}$ be two continuous and comonotone functions. Then the inequality
\begin{eqnarray}
\label{*}
-  \hspace{-1.1em}\int_0^1 f^s d\mu\cdot-  \hspace{-1.1em}\int_0^1 g^sd\mu\le -  \hspace{-1.1em}\int_0^1 \left(f\cdot g\right)^sd\mu,
\end{eqnarray}
holds, where $s>1$ and $\mu$ be the Lebesque measure on $\mathbb{R}$.
\end{theorem}\begin{proof}
Let $f, g: [0, 1]\to\mathbb{R}$ be two strictly increasing functions and let
$f^s=F$, $g^s=G$, $(f\cdot g)^s=H$,
$-  \hspace{-1.1em}\int_0^1 f^sd\mu=p$,
$-  \hspace{-1.1em}\int_0^1g^sd\mu=q$ and
$-  \hspace{-1.1em}\int_0^1 (f\cdot g)^sd\mu=r$. We know $p, q\in[0, 1]$ and $0<r\le 1$. In  the first case we let $p, q\in(0, 1)$ and $0<r<1$.
$$\left\{\begin{array}{ll}
pq\le p\Rightarrow 1-pq\ge 1-p\Rightarrow F(1-pq)\ge F(1-p)=p\\ pq\le q\Rightarrow 1-pq\ge 1-q\Rightarrow G(1-pq)\ge G(1-q)=q.
\end{array}\right.$$
If we suppose that $r<pq$, we get
$$1-r>1-pq\Rightarrow\left\{\begin{array}{ll}
F(1-r)>F(1-pq)\ge F(1-p)\ge p\\ G(1-r)>G(1-pq)\ge G(1-q)\ge q
\end{array}\right.\Rightarrow r>pq.$$
This relation is in contradiction with
$$-  \hspace{-1.1em}\int_0^1(f\cdot g)^sd\mu=-  \hspace{-1.1em}\int_0^1 Hd\mu=r\Rightarrow H(1-r)=r.$$
If we assume $p=1$ or $q=1$, the relation $r\ge pq$ is valid.

Suppose $p=1$. Therefore
$$F(1-p)\ge p\Rightarrow F(0)\ge 1.$$
Based on $F\uparrow$, on interval $[0, 1]$ we have $F\ge 1$. Therefore
$$-  \hspace{-1.1em}\int_0^1 F\cdot Gd\mu\ge -  \hspace{-1.1em}\int_0^1 Fd\mu\cdot -  \hspace{-1.1em}\int_0^1 Gd\mu.$$
Now, suppose that $f, g:[0, 1]\to\mathbb{R}$ are two strictly decreasing functions. By assumption at the first case, if $p, q\in(0, 1)$ and $0<r<1$, then we have 
$$\left\{\begin{array}{ll}pq\le p\Rightarrow F(pq)\ge F(p)\ge p\\ pq\le q\Rightarrow G(pq)\ge G(q)\ge q.
\end{array}\right.$$
If we let $r<pq$, we have
$$\left\{\begin{array}{ll}F(r)>F(pq)\ge F(p)\ge p\\ G(r)>G(pq)\ge G(q(\ge q
\end{array}\right.\Rightarrow H(r)=F(r)\cdot G(r)>pq,$$
that which contradicts with 
$$(f\cdot g)(r)=H(r)=r.$$
If we assume that $p=1$ or $q=1$, the inequality $r\ge pq$ is valid.

Assume $p=1$. Then
$$F(p)\ge p\Rightarrow F(1)\ge 1 .$$
Based on $F\downarrow$ on interval $[0, 1]$, hence $F\ge 1$ and therefore
$$FG\ge G\Rightarrow -  \hspace{-1.1em}\int_0^1 F\cdot G d\mu\ge -  \hspace{-1.1em}\int_0^1 Gd\mu=q=p\cdot q.$$
Thereby
$$-  \hspace{-1.1em}\int_0^1 F\cdot Gd\mu\ge -  \hspace{-1.1em}\int_0^1 Fd\mu\cdot -  \hspace{-1.1em}\int_0^1 Gd\mu.$$
The proof is now complete.
\end{proof}
In the following, we generalize the Theorem \ref{t2.33}, where the measure is the arbitrary fuzzy measure.

\begin{theorem} Let $f, g:[0, 1]\to\mathbb{R}$ be two continuous and comonoton functions. Then the inequality
$$-  \hspace{-1.1em}\int_0^1 f^s d\mu\cdot -  \hspace{-1.1em}\int_0^1 g^sd\mu\le -  \hspace{-1.1em}\int_0^1(f\cdot g)^sd\mu,$$
holds, where $s>1$ and $\mu$ is an arbitrary fuzzy measure on $X=[0, 1]$.
\end{theorem}\begin{proof}
We take
$$\left\{\begin{array}{lll}-  \hspace{-1.1em}\int_0^1 f^sd\mu=-  \hspace{-1.1em}\int_0^1\mu\left([0, 1]\cap F_\alpha\right)dm,\\\\ -  \hspace{-1.1em}\int_0^1 g^sd\mu=-  \hspace{-1.1em}\int_0^1 \mu\left([0, 1]\cap G_\alpha\right)dm,\\ \\-  \hspace{-1.1em}\int_0^1\left(f\cdot g\right)^sd\mu=-  \hspace{-1.1em}\int_0^1\mu\left(f\cdot g\right)^sd\mu=-  \hspace{-1.1em}\int_0^1\mu\left([0, 1]\cap H_\alpha\right)dm,\end{array}\right.$$
where
$F_\alpha=\left\{x\in\mathbb{R}|f^s(x)\ge\alpha\right\}$,
$G_\alpha=\left\{x\in\mathbb{R}|g^s(x)\ge\alpha\right\}$,
$H_\alpha=\left\{x\in\mathbb{R}|(f\cdot g) ^s(x)\ge\alpha\right\}$,$\alpha\ge 0$ and $m$ is a Lebesque measure on $\mathbb{R}$. If we let
$A(\alpha):=\mu\left([0, 1]\cap F_\alpha\right)$,
$B(\alpha):=\mu\left([0, 1]\cap B_\alpha\right)$ and
$C(\alpha):=\mu\left([0, 1]\cap H_\alpha\right)$, 
Because $A(\alpha)$, $B(\alpha)$ and $C(\alpha)$ respect to $\alpha$ are nondecreasing and $m$ is a Lebesque measure, then we have
$$-  \hspace{-1.1em}\int_0^1 A(\alpha)dm\cdot -  \hspace{-1.1em}\int_0^1 B(\alpha)dm\le -  \hspace{-1.1em}\int_0^1 C(\alpha)dm.$$
Since
$\dfrac{(M+m)^2}{4Mm}\ge 1$, hence
$$-  \hspace{-1.1em}\int_0^1A(\alpha)dm\cdot -  \hspace{-1.1em}\int_0^1 B(\alpha)dm\le\dfrac{(M+m)^2}{4Mm}-  \hspace{-1.1em}\int_0^1 C(\alpha)dm.$$
The proof is now complete.
\end{proof}

\begin{example}
Let $f, g:[0, 1]\to\mathbb{R}$ defined by
$f(x)=\dfrac{1}{2}\sqrt{x}$, $g(x)=\dfrac{1}{4}x$, $h(x)=f(x)\cdot g(x)$, $s=2$ and $\mu$ be a Lebesque measure on $\mathbb{R}$. With simple calculations, we get
\begin{eqnarray*}
&& -  \hspace{-1.1em}\int_0^1 f^2d\mu=\sup_{\alpha\in[0, \frac{1}{4}]}\left[\alpha\wedge(1-4\alpha)\right]=\dfrac{1}{5}=0.2,\\&&
-  \hspace{-1.1em}\int_0^1 g^2d\mu=\sup_{\alpha\in[0, \frac{1}{16}]}\left[\alpha\wedge(1-4\sqrt{\alpha})\right]=\dfrac{18-\sqrt{320}}{2}=0.05,\\&&
-  \hspace{-1.1em}\int_0^1 (f\cdot g)^2d\mu=-  \hspace{-1.1em}\int_0^1 h^2d\mu=\sup_{\alpha\in[0, \frac{1}{64}]}\left[\alpha\wedge(1-4\sqrt[3]{\alpha})\right]=0.015 .
\end{eqnarray*}
Therefore, we get
$$0.015=-  \hspace{-1.1em}\int_0^1 h^2d\mu\ge -  \hspace{-1.1em}\int_0^1 f^2d\mu\cdot-  \hspace{-1.1em}\int_0^1 g^d\mu=0.2\cdot 0.05.$$
\end{example}





\begin{example}
Let $X=[0, 1]$ and the fuzzy measure $\mu$ be defined as $\mu(A)=m^2(A)$, where $m$ is the Lebesque measure. And let 
$$f(x)=\left\{\begin{array}{lll}
\sqrt{x},\quad x\in[0, \frac{1}{4}]\\ \dfrac{\sqrt{2}}{2},\quad x\in(\frac{1}{4}, \frac{1}{2})\\ \sqrt{x},\quad x\in[\frac{1}{2}, 1]
\end{array}\right.\Rightarrow f^2(x)=\left\{\begin{array}{lll}x,\quad x\in[0, \frac{1}{4}]\\ \dfrac{1}{2},\quad x\in(\frac{1}{4}, \frac{1}{2})\\ x,\quad x\in[\frac{1}{2}, 1]
\end{array}\right.$$
and 
$$g(x)=\left\{\begin{array}{ll}x,\quad x\in[0, \frac{1}{2})\\ \sqrt{x},\quad x\in[\frac{1}{2}, 1]
\end{array}\right.\Rightarrow g^2(x)=\left\{\begin{array}{ll}
x^2,\quad x\in[0, \frac{1}{2})\\ x,\quad x\in[\frac{1}{2}, 1]
\end{array}\right.$$
Then 
\begin{eqnarray*}
&& -  \hspace{-1.1em}\int_0^1 f^2d\mu=\bigvee\left(\alpha\wedge m^2\left([0, 1]\cap F_\alpha\right)\right)=\dfrac{1}{2},\\&& 
-  \hspace{-1.1em}\int_0^1 g^2d\mu=\bigvee\left(\alpha\wedge m^2\left([0, 1]\cap G_\alpha\right)\right)=\dfrac{1}{4},\\&& 
\dfrac{1}{4}=-  \hspace{-1.1em}\int_0^1\left(f\cdot g\right)^2 d\mu>\left(-  \hspace{-1.1em}\int_0^1 f^2d\mu\right)\left(-  \hspace{-1.1em}\int_0^1 g^2d\mu\right)=\dfrac{1}{2}\cdot\dfrac{1}{4}.
\end{eqnarray*}
\end{example}

\begin{remark}
If $f, $ and $g$ are countermonotone functions, 
In the fuzzy case, 
 it may not change in an unequal direction. This is illustrated by an example in the following.
\end{remark}


\begin{example}
Let $f, g:[0, 1]\to\mathbb{R}$, 
$f(x)=x+1$, $g(x)=\dfrac{1}{x+1}$
and $\mu$ is the Lebesque measure on $\mathbb{R}$ and $s=3$. With a simple calculation, we have
\begin{eqnarray*}
&& f^3(x)=(x+1)^3=F,\\&&
g^3(x)=\left(\dfrac{1}{x+1}\right)^3=G,\\&&
(f\cdot g)^3(x)=1=H.
\end{eqnarray*}
Then 
\begin{eqnarray*}
&& -  \hspace{-1.1em}\int_0^1 Fd\mu=\sup\left[\alpha\wedge\left(-\sqrt{3}[\alpha]\right)\right)=1,\\&& 
-  \hspace{-1.1em}\int_0^1 Gd\mu=\sup\left[\alpha\wedge\left(\dfrac{1}{\sqrt{3}[\alpha]}-1\right)\right]=0.38,\\&& 
-  \hspace{-1.1em}\int_0^1 Hd\mu=-  \hspace{-1.1em}\int_0^1 1d\mu=1 .
\end{eqnarray*}
Therefore 
$$1=-  \hspace{-1.1em}\int_0^1 Hd\mu\ge -  \hspace{-1.1em}\int_0^1 Fd\mu\cdot-  \hspace{-1.1em}\int_0^1 Gd\mu=1\times 0. 38.$$
\end{example}

\section{Diyaz-Matcalf inequality for pseudo integrals}

In this Section, we state and prove Diaz-Metcalf inequality for the pseudo integral.

\begin{theorem}
Let $f, h:[0, 1]\to[a, b]$ be two measurable comonotone functions and $g:[a, b]\to[0, \infty)$ be a continuous and monotone function. Then the inequality
\begin{eqnarray}
\label{new}
\int_{[0, 1]}^\oplus (f\odot h)^s dx\ge \int_{[0, 1]}^\oplus f^s dx\odot \int_{[0, 1]}^\oplus h^sdx,
\end{eqnarray}
holds, where $s>1$.
\end{theorem}\begin{proof}
\begin{eqnarray*}
\int_{[0, 1]}^\oplus (f\odot h)^sdx &=&
\int_{[0, 1]}^\oplus\left(f^s\odot h^s\right)dx\\&=&
\int_{[0, 1]}^\oplus \left(f^s\odot h^s\right)dx\\&=&
g^{-1}\left[\int_0^1 g\left(f^s\odot h^s\right)dx\right]\\&=&
g^{-1}\int_0^1 g(f^s)\times g(g^s) dx\\&\ge &
g^{-1}\left[\int_0^1 g(f^s)dx\times \int_0^1 g(h^s)dx\right]\\&=&
g^{-1}\left[g\left(g^{-1}\left(\int_0^1 g(f^s)dx\right)\right)\times g\left(g^{-1}\left(\int_0^1 g(h^s)dx\right)\right)\right]\\&=&
g^{-1}\left(g\left(\int_{[0, 1]}^\oplus f^sdx\right)\times g\left(\int_{[0, 1]}^\oplus h^sdx\right)\right)\\&=&
\int_{[0, 1]}^\oplus f^sdx\odot \int_{[0, 1]}^\oplus h^sdx.
\end{eqnarray*}
The proof is now complete.
\end{proof}

\begin{example}
Let $f, h: [0, 1]\to [0, 1]$ defined by $f(x)=\sqrt{x}$ and $h(x)=\sqrt{\dfrac{1}{2}x}$. Let $g:[0, 1]\to'0, \infty)$ be $g(x)=x^2$ and $s=2$. A strightforward calculation shows that 
$$\dfrac{\sqrt{20}}{2}=\int_{[0, 1]}^\oplus \left(f^2\odot h^2\right)dx\ge\int_{[0, 1]}^\oplus f^2dx\odot\int_{[0, 1]}^\oplus h^2dx=\dfrac{1}{6}.$$
\end{example}

Now, we generalize the Diaz-Metcalf type inequlity by the semiring $\left([0, 1], \max, \odot\right)$, where $\odot$ is generated.

\begin{theorem}
 Let $f, h: [0, 1]\to [a, b]$
be two measurable comonotone functions and $([0, 1], \sup, \odot)$ be a simiring and $m$ be the same as Theorems \ref{t1} and \ref{t2}. If $g$ is the continuous and bounded function, then the following inequality is holds
$$\int_{[0, 1]}^{\sup} (f\odot h)^s dx\ge \left( \int_{[0, 1]}^{\sup} f^s \odot dm\right)\odot \left(\int_{[0, 1]}^{\sup} h^s \odot dm\right),$$
where $f$ and $h$ are comonoton functions and $s>1$.
\end{theorem}\begin{proof}
We have
\begin{eqnarray*}
\int_{[0, 1]}^{\sup} (f\odot h)^s \odot dm &=&
\lim_{\lambda\to\infty} (f\odot h)^s\odot dm_\lambda\\&=& \lim_{\lambda\to\infty} \left(g^\lambda\right)^{-1}\left(\int_0^1 g^\lambda\left(f^s\odot h^s\right)dx\right)\\&=&
\lim_{\lambda\to\infty}\left(g^\lambda\right)^{-1}\left( \int_0^1 g^\lambda (f^s\odot h^s)(x)dx\right)\\&\ge &
\lim_{\lambda\to\infty} \left(\left(g^\lambda\right)^{-1}\left(\int_0^1 g^\lambda(f^s(x))dx\right)\odot \left(g^\lambda\right)^{-1}\left(\int_0^1 g^\lambda(h^s(x))dx\right)\right)\\&=&
\left(\lim_{\lambda\to\infty}\left(g^\lambda\right)^{-1}\left(\int_0^1 g^\lambda(f^s(x)\right)dx\right)\odot \left(\lim_{\lambda\to\infty}\left(g^\lambda\right)^{-1}\left(\int_0^1 g^\lambda(h^s(x))dx\right)\right)\\&=&
\left(\int_{[0, 1]}^{\sup} f^s\odot dm\right)\odot \left(\int_{[0, 1]}^{\sup} h^s\odot dm\right).
\end{eqnarray*}
The proof is now complete.
\end{proof}

\begin{example}
Let 
$g^\lambda(x)=e^{\lambda x}$ and $\psi(x)$ be the form Theorem \ref{t2}. Then 
$$x\oplus y=\lim_{\lambda\to\infty}\dfrac{1}{\lambda}\ln\left(e^{\lambda x}+e^{\lambda y}\right)=\max(x, y)$$
and 
$$x\odot y=\lim_{\lambda\to\infty}\dfrac{1}{\lambda} \ln\left(e^{\lambda x}\cdot e^{\lambda y}\right)=x+y.$$
Therefore, we get 
$$\sup\left(f^s(x)+h^s(x)\psi(x)\right)\ge\sup\left(f^s(x)+\psi(x)\right)+\sup\left(h^s(x)+\psi(x)\right).$$
\end{example}

\section{Further discussion}
\begin{theorem}
Let $f, h: [0, 1]\to [a, b]$ be two measurable functions. Let a generator $g:[a, b]\to[0, \infty)$ of the pseudo addition $\oplus$ and the pseudo-multiplication $\odot$ be an increasing function. Let $\varphi: [a, b]\to [a, b]$ be a continuous and strictly increasing function such that commutes with $\odot$. If $f$ and $h$ are comonotone, then the inequality
\begin{eqnarray*}
\varphi^{-1}\left(\int_{[0, 1]}^{\oplus} \varphi\left(f\odot h\right)^sdx\right)
&\ge &
\left(\varphi^{-1}\left(\int_{[0, 1]}^{\oplus} \varphi(f^s)dx\right)\right)\odot\left(\varphi^{-1}\left(\int_{[0, 1]}^{\oplus} \varphi(h^s)dx\right)\right),
\end{eqnarray*}
holds, where $s>1$.
\end{theorem}\begin{proof}
Since $\varphi$ commutes with $\odot$, then we have
\begin{eqnarray}
\label{(i)}
\int_{[0, 1]}^{\oplus}\varphi\left(f^s\odot h^s\right)dx=\int_{[0, 1]}^{\oplus}\left(\varphi(f^s)\odot \varphi(h^s)\right)dx.
\end{eqnarray}
If $f$ and $h$ are comonotone functions and $\varphi$ is a continuous and strictly function, then $\varphi(f^s)$ and $\varphi(h^s)$ are also comonotone. From \eqref{(i)} and using Chebyshev's Theorem, we have
\begin{eqnarray*}
\int_{[0, 1]}^{\oplus}\varphi\left(f^s\odot h^s\right)dx&\ge& \left(\int_{[0, 1]}^{\oplus}\varphi(f^s)dx\right)\odot \left(\int_{[0, 1]}^{\oplus} \varphi(h^s)dx\right).
\end{eqnarray*}
It follows that 
\begin{eqnarray*}
\varphi^{-1}\int_{[0, 1]}^{\oplus}\varphi\left(f^s\odot h^s\right)dx&\ge& \varphi^{-1}\left[\left(\int_{[0, 1]}^{\oplus}\varphi(f^s)dx\right)\odot \left(\int_{[0, 1]}^{\oplus} \varphi(h^s)dx\right)\right]\\&=&
\left[\left(\varphi^{-1}\left(\int_{[0, 1]}^{\oplus}\varphi(f^s)dx\right)\right)\odot\left(\varphi^{-1}\left(\int_{[0, 1]}^{\oplus}\varphi(h^s)dx\right)\right)\right],
\end{eqnarray*}
where $\varphi$ commutes with $\odot$, hence \eqref{(i)} is valid.

Similarly, if $f$ and $h$ are countermonotone functions and $\varphi$ is a continuous and strictly increasing function, then $\varphi(f^s)$ and $\varphi(h^s)$ are also countermonotone.

From \eqref{(i)} and using Chebyshev's Theorem, we have
$$\varphi^{-1}\left(\int_{[0, 1]}^{\oplus} \varphi\left(f^s\odot h^s\right)dx\right)\le\left(\varphi^{-1}\left(\int_{[0, 1]}^{\oplus} \varphi(f^s)dx\right)\right)\odot\left(\varphi^{-1}\left(\int_{[0, 1]}^{\oplus} \varphi(h^s)dx\right)\right),$$
where $\varphi$ commutes with $\odot$. Thereby, the theorem is proved.
\end{proof}


Now, we consider the second class, when $x\oplus y=\max(x, y)$ and $x\odot y=g^{-1}\left(g(x)g(y)\right)$.

\begin{theorem}
Let $f, h: [0, 1]\to [a, b]$ be two continuous functions and $\odot$ be represented by an increasing multiplicative generator $g$ and $\varphi: [a, b]\to [a, b]$ be a continuous and strictly increasing functions such that $\varphi$ commutes with $\odot$ be the same as in Theorems \ref{t1} and \ref{t2}. If $f$ and $h$ are comonotone, then the inequality
\begin{eqnarray}
\label{(ii)}
\varphi^{-1} \left(\int_{[0, 1]}^{\sup} \varphi\left(f^s\odot h^s\right)\odot dm\right)\ge \left(\varphi^{-1}\left(\int_{[0, 1]}^{\sup} \varphi(f^s)\odot dm\right)\right)\odot \left(\varphi^{-1}\left(\int_{[0, 1]}^{\sup}\varphi(f^s)\odot dm\right)\right),
\end{eqnarray}
holds, where $s>1$.
\end{theorem}\begin{proof}
Theorem \ref{t2}, implies that
\begin{eqnarray*}
\int_{[0, 1]}^{\sup}\varphi\left(f^s\odot h^s\right)\odot dm &=& \lim_{\lambda\to\infty} \int_{[0, 1]}^{\oplus_\lambda} \varphi\left(f^s\odot h^s\right)\odot dm_{\lambda}\\&=&
\lim_{\lambda\to\infty}\left(g^{\lambda}\right)^{-1}\left(\int_0^1 g^\lambda\left(\left(f^s\odot h^s\right)\right)(x)dx\right).
\end{eqnarray*}
If $f$ and $h$ are comonotone functions and $\varphi$ is a continuous and strictly functions, then by using first Theorem \ref{t4.1} and the Theorem \ref{t2}, we have
\begin{eqnarray*}
&&\varphi^{-1}\left(\int_{[0, 1]}^{\sup}\varphi\left(f^s\odot h^s\right)\odot dm\right)\\&=&
\varphi^{-1}\left(\lim_{\lambda\to\infty} \left(g^{\lambda}\right)^{-1}\left(\int_0^1 g^\lambda\left(f^s\odot h^s\right)(x)\right)dx\right)\\&=&
\lim_{\lambda\to\infty} \left(\varphi^{-1}\left(\left(g^{\lambda}\right)^{-1}\left(\int_0^1 g^{\lambda}\left(\varphi\left(f^s\odot h^s\right)(x)\right)dx\right)\right)\right) \\&\ge &
\lim_{\lambda\to\infty}\left[\left(\left(g^{\lambda}\right)^{-1}\left(\varphi^{-1}\left(\int_0^1 g^\lambda \left(\varphi(f^s(x)\right)\right)\right)\right) \odot \left(\left(g^{\lambda}\right)^{-1}\left(\varphi^{-1}\left(\int_0^1 g^{\lambda} \left(\varphi(h^s(x)\right)\right)\right)\right)\right]\\&=&
\lim_{\lambda\to\infty}\left(
\left(g^{\lambda}\right)^{-1}
\left(
\varphi^{-1}\left(
\int_0^1 g^\lambda \left(
\varphi\left(f^s(x)\right)
\right)\right)dx\right)\right)\odot \lim_{\lambda\to\infty}\left(\left(g^{\lambda}\right)^{-1}\left(\varphi^{-1}\left(\int_0^1 g^\lambda \left(\varphi\left(h^s(x)\right)\right)\right)dx\right)\right)\\&=&
\lim_{\lambda\to\infty}\left(g^{\lambda}\right)^{-1}\left(\int_0^1 g^\lambda\left(\left(f^s\odot h^s\right)\right)(x)dx\right),
\end{eqnarray*}
where $\varphi$ commutes with $\odot$. Hence \eqref{(ii)} is valid.

\end{proof}

\date{\scriptsize $^{a}$
E-mail: mkmk01@gmail.com,}
\date{\scriptsize $^{b}$bdaraby@maragheh.ac.ir
E-mail: ,}
\date{\scriptsize $^{c}$
E-mail: rahimi@maragheh.ac.ir
}

\begin{thebibliography}{5}
\begin{scriptsize}

\bibitem{aga}
Agahi, H., Mesiar, R., \& Ouyang, Y. (2010). Chebyshev type inequalities for pseudo-integrals. Nonlinear Analysis: Theory, Methods \& Applications, 72(6), 2737-2743.











\bibitem{2-DDD} 	
Daraby, B. (2012). Generalization of the Stolarsky type inequality for pseudo-integrals. Fuzzy Sets and Systems, 194, 90-96.

\bibitem{a}
Daraby, B., \& Arabi, L. (2013). Related Fritz Carlson type inequalities for Sugeno integrals. Soft Computing, 17(10), 1745-1750.


\bibitem{b}
Daraby, B., \& Ghadimi, F. (2014). General Minkowski type and related inequalities for seminormed fuzzy integrals. Sahand Communications in Mathematical Analysis, 1(1), 9-20.

\bibitem{dar2}
Daraby, B., Asll, H. G., \& Sadeqi, I. (2017). General related inequalities to Carlson-type inequality for the Sugeno integral. Applied Mathematics and Computation, 305, 323-329.


\bibitem{bre}
Daraby, B., Rostampour, F., \& Rahimi, A. (2015). Hardy's type inequality for pseudo integrals. Acta Universitatis Apulensis, 42, 53-65.


\bibitem{c}
Daraby, B., Shafiloo, A., \& Rahimi, A. (2015). Generalizations of the Feng Qi type inequality for Pseudo-integral. Gazi University Journal of Science, 28(4), 695-702.


\bibitem{di}
J. B. Diayz, F. T. Metcalf. (2011). {\em A Diaaz-Metcalf inequality for positive linear maps and its applications}. ELA 22, 179-190


\bibitem{p4}
Ezghari, S., Zahi, A., \& Zenkouar, K. (2017). A new nearest neighbor classification method based on fuzzy set theory and aggregation operators. Expert Systems with Applications, 80, 58-74.


\bibitem{5}
Flores-Franulič, A., \& Román-Flores, H. (2007). A Chebyshev type inequality for fuzzy integrals. Applied Mathematics and Computation, 190(2), 1178-1184.


\bibitem{6}
Flores-Franulič, A., Román-Flores, H., \& Chalco-Cano, Y. (2009). Markov type inequalities for fuzzy integrals. Applied Mathematics and Computation, 207(1), 242-247.











 \bibitem{mes}
  Mesiar, R., \& Pap, E. (1999). Idempotent integral as limit of g-integrals. Fuzzy sets and systems, 102(3), 385-392.









\bibitem{12}
Ralescu, D., \& Adams, G. (1980). The fuzzy integral. Journal of Mathematical Analysis and Applications, 75(2), 562-570.



\bibitem{p2}
Rodger, J. A., \& George, J. A. (2017). Triple bottom line accounting for optimizing natural gas sustainability: A statistical linear programming fuzzy ILOWA optimized sustainment model approach to reducing supply chain global cybersecurity vulnerability through information and communications technology. Journal of cleaner production, 142, 1931-1949.




 





\bibitem{20}
Sugeno, M. (1974). Theory of fuzzy integrals and its applications. Doct. Thesis, Tokyo Institute of technology.


\bibitem{21}
Wang, Z., \& Klir, G. J. (2013). Fuzzy measure theory. Springer Science \& Business Media.


\bibitem{Ouyang3}
Ouyang, Y., \& Fang, J. (2008). Sugeno integral of monotone functions based on Lebesgue measure. Computers \& Mathematics with Applications, 56(2), 367-374.
\end{scriptsize}

\end{thebibliography}
\end{document}